\documentclass[11pt,fleqn]{article}

\usepackage{amsmath,amssymb,amsthm}
\usepackage{geometry}
\usepackage[pdfpagemode=UseNone,pdfstartview=FitH]{hyperref}

\newcommand{\abs}[1]{\left|#1\right|}
\newcommand{\bdry}[1]{\partial #1}
\newcommand{\closure}[1]{\overline{#1}}
\newcommand{\dint}{\ds{\int}}
\newcommand{\ds}[1]{\displaystyle #1}
\newcommand{\eps}{\varepsilon}
\newcommand{\half}{\frac{1}{2}}
\newcommand{\id}[1][]{id_{\, #1}}
\newcommand{\norm}[2][]{\left\|#2\right\|_{#1}}
\renewcommand{\O}{\text{O}}
\renewcommand{\o}{\text{o}}
\newcommand{\PS}[1]{$(\text{PS})_{#1}$}
\newcommand{\pnorm}[2][]{\if #1'' \left|#2\right|_p \else \left|#2\right|_{#1} \fi}
\newcommand{\R}{\mathbb R}
\newcommand{\restr}[2]{\left.#1\right|_{#2}}
\newcommand{\seq}[1]{\left(#1\right)}
\newcommand{\set}[1]{\left\{#1\right\}}

\newenvironment{enumroman}{\begin{enumerate}

}{\end{enumerate}}

\newtheorem{lemma}{Lemma}[section]
\newtheorem{proposition}[lemma]{Proposition}
\newtheorem{theorem}[lemma]{Theorem}

\theoremstyle{remark}
\newtheorem{remark}[lemma]{Remark}

\numberwithin{equation}{section}

\title{\bf On critical elliptic problems with singular Trudinger-Moser nonlinearities}
\author{\bf Shiqiu Fu and Kanishka Perera\\
Department of Mathematical Sciences\\
Florida Institute of Technology\\
Melbourne, FL 32901, USA\\
\em sfu2013@my.fit.edu \& kperera@fit.edu}
\date{}

\begin{document}

\maketitle

\begin{abstract}
We establish some existence results for a class of critical elliptic problems with singular exponential nonlinearities. We do not assume any global sign conditions on the nonlinearity, which makes our results new even in the nonsingular case.
\end{abstract}

\section{Introduction}

The purpose of this paper is to establish some existence results for the class of singular elliptic problems with exponential nonlinearities \begin{equation} \label{1.1}
\left\{\begin{aligned}
- \Delta u & = h(u)\, \frac{e^{\alpha u^2}}{|x|^\gamma} && \text{in } \Omega\\[5pt]
u & = 0 && \text{on } \bdry{\Omega},
\end{aligned}\right.
\end{equation}
where $\Omega$ is a bounded Lipschitz domain in $\R^2$ containing the origin, $\alpha > 0$, $0 \le \gamma < 2$, and $h$ is a continuous function such that
\begin{equation} \label{1.1.1}
\lim_{|t| \to \infty}\, h(t) = 0
\end{equation}
and
\begin{equation} \label{1.2}
0 < \beta := \liminf_{|t| \to \infty}\, th(t) < \infty.
\end{equation}
This problem is motivated by the following singular Trudinger-Moser embedding:
\[
\int_\Omega \frac{e^{\alpha u^2}}{|x|^\gamma}\, dx < \infty \quad \forall u \in H^1_0(\Omega)
\]
for all $\alpha > 0$ and $0 \le \gamma < 2$, and
\begin{equation} \label{1.3}
\sup_{\norm[H^1_0(\Omega)]{u} \le 1}\, \int_\Omega \frac{e^{\alpha u^2}}{|x|^\gamma}\, dx < \infty
\end{equation}
if and only if
\[
\frac{\alpha}{4 \pi} + \frac{\gamma}{2} \le 1
\]
(see Adimurthi and Sandeep \cite{MR2329019}). Problem \eqref{1.1} is critical with respect to this embedding and hence lacks compactness. The case $\beta = \infty$ was considered in \cite{MR2329019}, so we will focus on the case $0 < \beta < \infty$ here.

The nonsingular case $\gamma = 0$ has been widely studied in the literature (see, e.g., Adimurthi \cite{MR1079983}, Adimurthi and Yadava \cite{MR1044289}, de Figueiredo et al.\! \cite{MR1386960,MR1399846}, Marcos B. do {\'O} \cite{MR1392090}, de Figueiredo et al.\! \cite{MR1865413,MR2772124}, Zhang et al.\! \cite{MR2112476}, Perera and Yang \cite{MR3616328}, and their references). However, in all these results it is assumed that $h(t) > 0$ for all $t > 0$ and $h(t) < 0$ for all $t < 0$. We impose no such global sign conditions on the nonlinearity $h$, so our results here are new even in the nonsingular case.

The singular eigenvalue problem
\begin{equation} \label{1.4}
\left\{\begin{aligned}
- \Delta u & = \lambda\, \frac{u}{|x|^\gamma} && \text{in } \Omega\\[5pt]
u & = 0 && \text{on } \bdry{\Omega}
\end{aligned}\right.
\end{equation}
will play a major role in our results. The first eigenvalue of this eigenvalue problem is positive and is given by
\begin{equation} \label{1.5}
\lambda_1(\gamma) = \inf_{u \in H^1_0(\Omega) \setminus \set{0}}\, \frac{\dint_\Omega |\nabla u|^2\, dx}{\dint_\Omega \frac{u^2}{|x|^\gamma}\, dx}.
\end{equation}
Set
\[
G(t) = \int_0^t h(s)\, e^{\alpha s^2}\, ds,
\]
let $d$ be the radius of the largest open ball centered at the origin that is contained in $\Omega$, and let
\[
\kappa = \frac{(2 - \gamma)^2}{2d^{2 - \gamma}}.
\]
Our first result is the following theorem.

\begin{theorem} \label{Theorem 1.1}
Assume that $\alpha > 0$ and $0 \le \gamma < 2$ satisfy $\alpha/4 \pi + \gamma/2 \le 1$, $h$ satisfies \eqref{1.1.1} and \eqref{1.2}, $G$ satisfies
\begin{gather}
\label{1.6} G(t) \ge 0 \quad \text{for } t \ge 0,\\[7.5pt]
\label{1.7} G(t) \le \half\, (\lambda_1(\gamma) - \sigma_1)\, t^2 \quad \text{for } |t| \le \delta
\end{gather}
for some $\sigma_1, \delta > 0$, and
\begin{equation} \label{1.8}
\beta > \frac{\kappa}{\alpha}.
\end{equation}
Then problem \eqref{1.1} has a nontrivial solution.
\end{theorem}

This theorem is new even in the nonsingular case $\gamma = 0$. Indeed, the corresponding result for the nonsingular case is proved in de Figueiredo et al.\! \cite{MR1386960, MR1399846} and Marcos B. do {\'O} \cite{MR1392090} only assuming that $h(t) \ge 0$ for all $t \ge 0$. This implies our assumption \eqref{1.6}, but clearly \eqref{1.6} is weaker. Moreover, we can further weaken the assumption \eqref{1.6} if we assume a larger lower bound on $\beta$. We have the following theorem.

\begin{theorem} \label{Theorem 1.2}
Assume that $\alpha > 0$ and $0 \le \gamma < 2$ satisfy $\alpha/4 \pi + \gamma/2 \le 1$, $h$ satisfies \eqref{1.1.1} and \eqref{1.2}, $G$ satisfies
\begin{gather}
\label{1.9} G(t) \ge - \half\, \sigma_0\, t^2 \quad \text{for } t \ge 0,\\[7.5pt]
\notag G(t) \le \half\, (\lambda_1(\gamma) - \sigma_1)\, t^2 \quad \text{for } |t| \le \delta
\end{gather}
for some $\sigma_0, \sigma_1, \delta > 0$, and
\begin{equation} \label{1.10}
\beta > \begin{cases}
\dfrac{2 \kappa}{\alpha}\, \dfrac{e^{\sigma_0/\kappa}}{3 - e^{\sigma_0/\kappa}} & \text{if } \sigma_0 \le \kappa \log 2\\[15pt]
\dfrac{2 \kappa}{\alpha}\, e^{\sigma_0/\kappa} & \text{if } \sigma_0 > \kappa \log 2.
\end{cases}
\end{equation}
Then problem \eqref{1.1} has a nontrivial solution.
\end{theorem}

\begin{remark}
We note that Theorem \ref{Theorem 1.2} reduces to Theorem \ref{Theorem 1.1} when $\sigma_0 = 0$.
\end{remark}

Now let $\seq{\lambda_k(\gamma)}$ be the sequence of eigenvalues of the eigenvalue problem \eqref{1.4}, repeated according to multiplicity (see Proposition \ref{Proposition 2.2} in the next section). Our last result is the following theorem.

\begin{theorem} \label{Theorem 1.3}
Assume that $\alpha > 0$ and $0 \le \gamma < 2$ satisfy $\alpha/4 \pi + \gamma/2 \le 1$, $h$ satisfies \eqref{1.1.1} and \eqref{1.2}, and $G$ satisfies
\begin{gather}
\label{1.11} G(t) \ge \half\, (\lambda_{k-1}(\gamma) + \sigma_0)\, t^2 \quad \forall t,\\[7.5pt]
\label{1.12} G(t) \le \half\, (\lambda_k(\gamma) - \sigma_1)\, t^2 \quad \text{for } |t| \le \delta
\end{gather}
for some $k \ge 2$ and $\sigma_0, \sigma_1, \delta > 0$. Then there exists a constant $c > 0$ depending on $\Omega$, $\alpha$, $\gamma$, and $k$, but not on $\sigma_0$, $\sigma_1$, or $\delta$, such that if
\begin{equation} \label{1.13}
\beta > \frac{2 \kappa}{\alpha}\, e^{c/\sigma_0},
\end{equation}
then problem \eqref{1.1} has a nontrivial solution.
\end{theorem}

This theorem is also new even in the nonsingular case $\gamma = 0$. The corresponding result for the nonsingular case is proved in de Figueiredo et al.\! \cite{MR1386960, MR1399846} only under the additional assumption that $0 < 2G(t) \le th(t)\, e^{\alpha t^2}$ for all $t \in \R \setminus \set{0}$.

Proofs of Theorems \ref{Theorem 1.1} and \ref{Theorem 1.2} will be given in Section \ref{Section 3} and the proof of Theorem \ref{Theorem 1.3} will be given in Section \ref{Section 4}, after proving a suitable compactness property of the associated variational functional in the next section.

\section{Preliminaries}

Weak solutions of problem \eqref{1.1} coincide with critical points of the $C^1$-functional
\[
E(u) = \half \int_\Omega |\nabla u|^2\, dx - \int_\Omega \frac{G(u)}{|x|^\gamma}\, dx, \quad u \in H^1_0(\Omega).
\]
We recall that a \PS{c} sequence of $E$ is a sequence $\seq{u_j} \subset H^1_0(\Omega)$ such that $E(u_j) \to c$ and $E'(u_j) \to 0$. Proofs of Theorems \ref{Theorem 1.1} and \ref{Theorem 1.2} will be based on the following compactness result.

\begin{proposition} \label{Proposition 2.1}
Assume that $\alpha > 0$ and $0 \le \gamma < 2$ satisfy $\alpha/4 \pi + \gamma/2 \le 1$, and $h$ satisfies \eqref{1.1.1} and \eqref{1.2}. Then for all $c \ne 0$ with
\[
c < \frac{2 \pi}{\alpha} \left(1 - \frac{\gamma}{2}\right),
\]
every {\em \PS{c}} sequence of $E$ has a subsequence that converges weakly to a nontrivial solution of problem \eqref{1.1}.
\end{proposition}

\begin{proof}
Let $\seq{u_j} \subset H^1_0(\Omega)$ be a \PS{c} sequence of $E$. Then
\begin{equation} \label{2.1}
E(u_j) = \half \norm{u_j}^2 - \int_\Omega \frac{G(u_j)}{|x|^\gamma}\, dx = c + \o(1)
\end{equation}
and
\begin{equation} \label{2.2}
E'(u_j)\, u_j = \norm{u_j}^2 - \int_\Omega u_j\, h(u_j)\, \frac{e^{\alpha u_j^2}}{|x|^\gamma}\, dx = \o(\norm{u_j}).
\end{equation}
First we show that $\seq{u_j}$ is bounded in $H^1_0(\Omega)$. Multiplying \eqref{2.1} by $4$ and subtracting \eqref{2.2} gives
\[
\norm{u_j}^2 + \int_\Omega \left(u_j\, h(u_j)\, e^{\alpha u_j^2} - 4\, G(u_j)\right) \frac{dx}{|x|^\gamma} = 4c + \o(\norm{u_j} + 1),
\]
so it suffices to show that $th(t)\, e^{\alpha t^2} - 4G(t)$ is bounded from below. Let $0 < \eps \le \beta/5$. By \eqref{1.1.1} and \eqref{1.2}, for some constant $C_\eps > 0$,
\begin{equation} \label{2.5}
|G(t)| \le \eps\, e^{\alpha t^2} + C_\eps
\end{equation}
and
\begin{equation} \label{2.3}
th(t)\, e^{\alpha t^2} \ge (\beta - \eps)\, e^{\alpha t^2} - C_\eps
\end{equation}
for all $t$, and the desired conclusion follows.

Since $\seq{u_j}$ is bounded in $H^1_0(\Omega)$, a renamed subsequence converges to some $u$ weakly in $H^1_0(\Omega)$, strongly in $L^p(\Omega)$ for all $p \in [1,\infty)$, and a.e.\! in $\Omega$. We have
\begin{equation} \label{2.6}
E'(u_j)\, v = \int_\Omega \nabla u_j \cdot \nabla v\, dx - \int_\Omega v\, h(u_j)\, \frac{e^{\alpha u_j^2}}{|x|^\gamma}\, dx \to 0
\end{equation}
for all $v \in H^1_0(\Omega)$. By \eqref{1.1.1}, given any $\eps > 0$, there exists a constant $C_\eps > 0$ such that
\begin{equation} \label{2.7}
|h(t)\, e^{\alpha t^2}| \le \eps\, e^{\alpha t^2} + C_\eps \quad \forall t.
\end{equation}
By \eqref{2.2},
\[
\sup_j\, \int_\Omega u_j\, h(u_j)\, \frac{e^{\alpha u_j^2}}{|x|^\gamma}\, dx < \infty,
\]
which together with \eqref{2.3} gives
\begin{equation} \label{2.8}
\sup_j\, \int_\Omega \frac{e^{\alpha u_j^2}}{|x|^\gamma}\, dx < \infty.
\end{equation}
For $v \in C^\infty_0(\Omega)$, it follows from \eqref{2.7} and \eqref{2.8} that the sequence $(v\, h(u_j)\, e^{\alpha u_j^2}/|x|^\gamma)$ is uniformly integrable and hence
\[
\int_\Omega v\, h(u_j)\, \frac{e^{\alpha u_j^2}}{|x|^\gamma}\, dx \to \int_\Omega v\, h(u)\, \frac{e^{\alpha u^2}}{|x|^\gamma}\, dx
\]
by Vitali's convergence theorem, so it follows from \eqref{2.6} that
\[
\int_\Omega \nabla u \cdot \nabla v\, dx - \int_\Omega v\, h(u)\, \frac{e^{\alpha u^2}}{|x|^\gamma}\, dx = 0.
\]
Then this holds for all $v \in H^1_0(\Omega)$ by density, so the weak limit $u$ is a solution of problem \eqref{1.1}.

Suppose that $u = 0$. Then
\[
\int_\Omega \frac{G(u_j)}{|x|^\gamma}\, dx \to 0
\]
since \eqref{2.5} and \eqref{2.8} imply that the sequence $\seq{G(u_j)/|x|^\gamma}$ is uniformly integrable, so \eqref{2.1} gives $c \ge 0$ and
\begin{equation} \label{2.9}
\norm{u_j} \to (2c)^{1/2}.
\end{equation}
Let $2c < \nu < 4 \pi\, (1 - \gamma/2)/\alpha$. Then $\norm{u_j} \le \nu^{1/2}$ for all $j \ge j_0$ for some $j_0$. Let $q = 4 \pi\, (1 - \gamma/2)/\alpha \nu > 1$ and let $1/(1 - 1/q) < r < 2/\gamma\, (1 - 1/q)$. By the H\"{o}lder inequality,
\[
\abs{\int_\Omega u_j\, h(u_j)\, \frac{e^{\alpha u_j^2}}{|x|^\gamma}\, dx} \le \left(\int_\Omega |u_j\, h(u_j)|^p\, dx\right)^{1/p}\! \left(\int_\Omega \frac{e^{q \alpha u_j^2}}{|x|^\gamma}\, dx\right)^{1/q}\! \left(\int_\Omega \frac{dx}{|x|^{\gamma r\, (1-1/q)}}\right)^{1/r},
\]
where $1/p + 1/q + 1/r = 1$. The first integral on the right-hand side converges to zero since $th(t)$ is bounded and $u = 0$, the second integral is bounded for $j \ge j_0$ by \eqref{1.3} since $q \alpha u_j^2 = 4 \pi\, (1 - \gamma/2)\, \widetilde{u}_j^2$, where $\widetilde{u}_j = u_j/\nu^{1/2}$ satisfies $\norm{\widetilde{u}_j} \le 1$, and the last integral is finite since $\gamma r\, (1 - 1/q) < 2$, so
\[
\int_\Omega u_j\, h(u_j)\, \frac{e^{\alpha u_j^2}}{|x|^\gamma}\, dx \to 0.
\]
Then $u_j \to 0$ by \eqref{2.2} and hence $c = 0$ by \eqref{2.9}, contrary to assumption. So $u$ is nontrivial.
\end{proof}

We close this preliminary section with a basic result for the singular eigenvalue problem \eqref{1.4}. Set $\omega(x) = |x|^{- \gamma}$ and let $L^2(\Omega,\omega)$ be the weighted Lebesgue space with the norm
\begin{equation} \label{2.10}
\pnorm[2,\, \omega]{u} = \left(\int_\Omega \omega(x)\, |u(x)|^2\, dx\right)^{1/2}.
\end{equation}
\begin{proposition} \label{Proposition 2.2}
The eigenvalues of the eigenvalue problem \eqref{1.4} are positive, have finite multiplicities, and form a nondecreasing sequence $\lambda_k(\gamma) \to \infty$. The space $L^2(\Omega,\omega)$ has an orthonormal basis consisting of eigenfunctions that are also orthogonal in $H^1_0(\Omega)$. Moreover, the eigenfunctions belong to $C^\alpha(\closure{\Omega})$ for some $\alpha \in (0,1)$.
\end{proposition}

\begin{proof}
Since $H^1_0(\Omega)$ is compactly embedded in $L^p(\Omega)$ for $1 \le p < \infty$ by the Sobolev embedding theorem and $L^p(\Omega) \hookrightarrow L^2(\Omega,\omega)$ for $p > 4/(2 - \gamma)$ by the H\"{o}lder inequality, $H^1_0(\Omega)$ is compactly embedded in $L^2(\Omega,\omega)$. The eigenvalue problem \eqref{1.4} can be written as
\[
Su = \lambda^{-1}\, u,
\]
where $S : L^2(\Omega,\omega) \to H^1_0(\Omega),\, f \mapsto u$ is the solution operator for the singular boundary value problem
\[
\left\{\begin{aligned}
- \Delta u & = \frac{f(x)}{|x|^\gamma} && \text{in } \Omega\\[5pt]
u & = 0 && \text{on } \bdry{\Omega}.
\end{aligned}\right.
\]
Since $S : L^2(\Omega,\omega) \to L^2(\Omega,\omega)$ is a compact symmetric operator, the first part of the proposition follows from the spectral theorem.

Let $u$ be an eigenfunction, let $1 < q < 2/\gamma$, and let $1 < s < 2/\gamma q$. By the H\"{o}lder inequality,
\[
\int_\Omega \bigg|\frac{u}{|x|^\gamma}\bigg|^q\, dx \le \left(\int_\Omega |u|^{qr}\, dx\right)^{1/r} \left(\int_\Omega \frac{dx}{|x|^{\gamma qs}}\right)^{1/s},
\]
where $1/r + 1/s = 1$. The first integral on the right-hand side is finite by the Sobolev embedding, and so is the second integral since $\gamma qs < 2$, so $u/|x|^\gamma \in L^q(\Omega)$. By the Calderon-Zygmund inequality, then $u \in W^{2,q}(\Omega)$, which is embedded in $C^\alpha(\closure{\Omega})$ for $\alpha = 2 - 2/q$ when $1 < q < 2$.
\end{proof}

\section{Proofs of Theorems \ref{Theorem 1.1} and \ref{Theorem 1.2}} \label{Section 3}

In this section we prove Theorems \ref{Theorem 1.1} and \ref{Theorem 1.2} by showing that the functional $E$ has the mountain pass geometry with the mountain pass level $c \in (0,2 \pi\, (1 - \gamma/2)/\alpha)$ and applying Proposition \ref{Proposition 2.1}.

\begin{lemma} \label{Lemma 3.1}
If \eqref{1.7} holds, then there exists a $\rho > 0$ such that
\[
\inf_{\norm{u} = \rho}\, E(u) > 0.
\]
\end{lemma}

\begin{proof}
Since \eqref{1.1.1} implies that $h$ is bounded, there exists a constant $C_\delta > 0$ such that
\[
|G(t)| \le C_\delta\, |t|^3\, e^{\alpha t^2} \quad \text{for } |t| > \delta,
\]
which together with \eqref{1.7} gives
\begin{equation} \label{3.1}
\int_\Omega \frac{G(u)}{|x|^\gamma}\, dx \le \half\, (\lambda_1(\gamma) - \sigma_1) \int_\Omega \frac{u^2}{|x|^\gamma}\, dx + C_\delta \int_\Omega |u|^3\, \frac{e^{\alpha u^2}}{|x|^\gamma}\, dx.
\end{equation}
By \eqref{1.5},
\begin{equation}
\int_\Omega \frac{u^2}{|x|^\gamma}\, dx \le \frac{\rho^2}{\lambda_1(\gamma)},
\end{equation}
where $\rho = \norm{u}$. Let $2 < r < 4/\gamma$. By the H\"{o}lder inequality,
\begin{equation} \label{3.3}
\int_\Omega |u|^3\, \frac{e^{\alpha u^2}}{|x|^\gamma}\, dx \le \left(\int_\Omega |u|^{3p}\, dx\right)^{1/p} \left(\int_\Omega \frac{e^{2 \alpha u^2}}{|x|^\gamma}\, dx\right)^{1/2} \left(\int_\Omega \frac{dx}{|x|^{\gamma r/2}}\right)^{1/r},
\end{equation}
where $1/p + 1/r = 1/2$. The first integral on the right-hand side is bounded by $C \rho^{3p}$ for some constant $C > 0$ by the Sobolev embedding. Since $2 \alpha u^2 = 2 \alpha \rho^2\, \widetilde{u}^2$, where $\widetilde{u} = u/\rho$ satisfies $\norm{\widetilde{u}} = 1$, the second integral is bounded when $\rho^2 \le 2 \pi\, (1 - \gamma/2)/\alpha$ by \eqref{1.3}. The last integral is finite since $\gamma r < 4$. So combining \eqref{3.1}--\eqref{3.3} gives
\[
\int_\Omega \frac{G(u)}{|x|^\gamma}\, dx \le \half \left(1 - \frac{\sigma_1}{\lambda_1(\gamma)}\right) \rho^2 + \O(\rho^3) \quad \text{as } \rho \to 0.
\]
Then
\[
E(u) \ge \half\, \frac{\sigma_1}{\lambda_1(\gamma)}\, \rho^2 + \O(\rho^3),
\]
and the desired conclusion follows from this for sufficiently small $\rho > 0$.
\end{proof}

We have $B_d(0) \subset \Omega$. For $j \ge 2$, let
\begin{equation} \label{3.4}
\omega_j(x) = \frac{1}{\sqrt{2 \pi}}\, \begin{cases}
\sqrt{\log j} & \text{if } |x| \le d/j\\[15pt]
\dfrac{\log\, (d/|x|)}{\sqrt{\log j}} & \text{if } d/j < |x| < d\\[15pt]
0 & \text{otherwise}.
\end{cases}
\end{equation}
It is easily seen that $\omega_j \in H^1_0(\Omega)$ with $\norm{\omega_j} = 1$. Moreover,
\begin{equation} \label{3.4.1}
\int_\Omega \frac{\omega_j}{|x|^\gamma}\, dx = \frac{d^{2 - \gamma}}{(2 - \gamma)^2}\, \sqrt{\frac{2 \pi}{\log j}} \left(1 - \frac{1}{j^{2 - \gamma}}\right)
\end{equation}
and
\begin{equation} \label{3.4.2}
\int_\Omega \frac{\omega_j^2}{|x|^\gamma}\, dx = \frac{2d^{2 - \gamma}}{(2 - \gamma)^3 \log j} \left[1 - \frac{(2 - \gamma) \log j + 1}{j^{2 - \gamma}}\right].
\end{equation}

\begin{lemma} \label{Lemma 3.2}
Assume that $h$ satisfies \eqref{1.1.1} and \eqref{1.2}.
\begin{enumroman}
\item \label{Lemma 3.2 (i)} For all $j \ge 2$, $E(t \omega_j) \to - \infty$ as $t \to \infty$.
\item \label{Lemma 3.2 (ii)} If \eqref{1.7} holds, then there exists a $j_0 \ge 2$ such that
    \begin{equation} \label{3.5}
    \sup_{t \ge 0}\, E(t \omega_{j_0}) < \frac{2 \pi}{\alpha} \left(1 - \frac{\gamma}{2}\right)
    \end{equation}
    in each of the following cases:
    \begin{enumerate}
    \item[{\em ({\em a})}] \eqref{1.6} and \eqref{1.8} hold,
    \item[{\em ({\em b})}] \eqref{1.9} and \eqref{1.10} hold.
    \end{enumerate}
\end{enumroman}
\end{lemma}

\begin{proof}
Fix $\eps > 0$. By \eqref{1.2}, $\exists M_\eps > 0$ such that
\begin{equation} \label{3.6}
th(t)\, e^{\alpha t^2} > (\beta - \eps)\, e^{\alpha t^2} \quad \text{for } |t| > M_\eps.
\end{equation}
Since $e^{\alpha t^2} > \alpha^2\, t^4/2$ for all $t$, then there exists a constant $C_\eps > 0$ such that
\begin{equation} \label{3.7}
th(t)\, e^{\alpha t^2} \ge \half\, (\beta - \eps)\, \alpha^2\, t^4 - C_\eps\, |t|
\end{equation}
and
\begin{equation} \label{3.8}
G(t) \ge \frac{1}{8}\, (\beta - \eps)\, \alpha^2\, t^4 - C_\eps\, |t|
\end{equation}
for all $t$. Since $\norm{\omega_j} = 1$ and $\omega_j \ge 0$, then
\[
E(t \omega_j) \le \frac{t^2}{2} - \frac{1}{8}\, (\beta - \eps)\, \alpha^2\, t^4 \int_\Omega \frac{\omega_j^4}{|x|^\gamma}\, dx + C_\eps\, t \int_\Omega \frac{\omega_j}{|x|^\gamma}\, dx,
\]
and \ref{Lemma 3.2 (i)} follows.

Set
\[
H_j(t) = E(t \omega_j) = \frac{t^2}{2} - \int_\Omega \frac{G(t \omega_j)}{|x|^\gamma}\, dx, \quad t \ge 0.
\]
If \ref{Lemma 3.2 (ii)} is false, then it follows from Lemma \ref{Lemma 3.1} and \ref{Lemma 3.2 (i)} that for all $j \ge 2$, $\exists t_j > 0$ such that
\begin{gather}
\label{3.9} H_j(t_j) = \frac{t_j^2}{2} - \int_\Omega \frac{G(t_j \omega_j)}{|x|^\gamma}\, dx = \sup_{t \ge 0}\, H_j(t) \ge \frac{2 \pi}{\alpha} \left(1 - \frac{\gamma}{2}\right),\\[7.5pt]
\label{3.10} H_j'(t_j) = t_j - \int_\Omega \omega_j\, h(t_j \omega_j)\, \frac{e^{\alpha t_j^2 \omega_j^2}}{|x|^\gamma}\, dx = 0.
\end{gather}
Since $G(t) \ge - C_\eps\, t$ for all $t \ge 0$ by \eqref{3.8}, \eqref{3.9} gives
\begin{equation} \label{3.11}
t_j^2 \ge t_0^2 - 2 \delta_j\, t_j,
\end{equation}
where
\[
t_0 = \sqrt{\frac{4 \pi}{\alpha} \left(1 - \frac{\gamma}{2}\right)}
\]
and
\[
\delta_j = C_\eps \int_\Omega \frac{\omega_j}{|x|^\gamma}\, dx \to 0 \quad \text{as } j \to \infty
\]
by \eqref{3.4.1}. First we will show that $t_j \to t_0$.

By \eqref{3.11}, $t_j \ge \sqrt{t_0^2 + \delta_j^2} - \delta_j$ and hence
\begin{equation} \label{3.13}
\liminf_{j \to \infty}\, t_j \ge t_0.
\end{equation}
Write \eqref{3.10} as
\begin{equation} \label{3.14}
t_j^2 = \int_{\set{t_j \omega_j > M_\eps}} t_j \omega_j\, h(t_j \omega_j)\, \frac{e^{\alpha t_j^2 \omega_j^2}}{|x|^\gamma}\, dx + \int_{\set{t_j \omega_j \le M_\eps}} t_j \omega_j\, h(t_j \omega_j)\, \frac{e^{\alpha t_j^2 \omega_j^2}}{|x|^\gamma}\, dx =: I_1 + I_2.
\end{equation}
Set $r_j = de^{- M_\eps \sqrt{2 \pi \log j}/t_j}$. Since $\liminf t_j > 0$, for all sufficiently large $j$, $d/j < r_j < d$ and $t_j \omega_j(x) > M_\eps$ if and only if $|x| < r_j$. So \eqref{3.6} gives
\begin{multline}
I_1 \ge (\beta - \eps) \int_{\set{|x| < r_j}} \frac{e^{\alpha t_j^2 \omega_j^2}}{|x|^\gamma}\, dx = (\underline{\beta} - \eps) \Bigg(\int_{\set{|x| \le d/j}} \frac{e^{\alpha t_j^2 \omega_j^2}}{|x|^\gamma}\, dx\\[7.5pt]
+ \int_{\set{d/j < |x| < r_j}} \frac{e^{\alpha t_j^2 \omega_j^2}}{|x|^\gamma}\, dx\Bigg) =: (\beta - \eps)\, (I_3 + I_4).
\end{multline}
We have
\begin{equation}
I_3 = e^{\alpha t_j^2 \log j/2 \pi} \int_{\set{|x| \le d/j}} \frac{dx}{|x|^\gamma} = \frac{2 \pi}{2 - \gamma} \left(\frac{d}{j}\right)^{2 - \gamma} j^{\alpha t_j^2/2 \pi} = \frac{2 \pi d^{2 - \gamma}}{2 - \gamma}\, j^{\alpha\, (t_j^2 - t_0^2)/2 \pi}.
\end{equation}
Since $th(t)\, e^{\alpha t^2} \ge - C_\eps\, t$ for all $t \ge 0$ by \eqref{3.7},
\begin{equation} \label{3.17}
I_2 \ge - C_\eps\, t_j \int_{\set{t_j \omega_j \le M_\eps}} \frac{\omega_j}{|x|^\gamma}\, dx \ge - \delta_j\, t_j.
\end{equation}
Combining \eqref{3.14}--\eqref{3.17} and noting that $I_4 \ge 0$ gives
\[
t_j^2 \ge (\beta - \eps)\, \frac{2 \pi d^{2 - \gamma}}{2 - \gamma}\, j^{\alpha\, (t_j^2 - t_0^2)/2 \pi} - \delta_j\, t_j.
\]
It follows from this that
\[
\limsup_{j \to \infty}\, t_j \le t_0,
\]
which together with \eqref{3.13} shows that $t_j \to t_0$.

Next we estimate $I_4$. We have
\begin{align}
I_4 & = \int_{\set{d/j < |x| < r_j}} \frac{e^{\alpha t_j^2\, [\log\, (d/|x|)]^2/2 \pi \log j}}{|x|^\gamma}\, dx \notag\\[7.5pt]
& = 2 \pi \left(\int_{d/j}^d e^{\alpha t_j^2\, [\log\, (d/r)]^2/2 \pi \log j}\, r^{1 - \gamma}\, dr - \int_{r_j}^d e^{\alpha t_j^2\, [\log\, (d/r)]^2/2 \pi \log j}\, r^{1 - \gamma}\, dr\right) \notag\\[7.5pt]
& = 2 \pi d^{2 - \gamma} \left(\log j \int_0^1 e^{- (2 - \gamma)\, t\, [1 - (t_j/t_0)^2\, t] \log j}\, dt - \int_{s_j}^1 s^{1 - \gamma}\, e^{\alpha t_j^2\, (\log s)^2/2 \pi \log j}\, ds\right), \label{3.19}
\end{align}
where $t = \log\, (d/r)/\log j$, $s = r/d$, and $s_j = r_j/d = e^{- M_\eps \sqrt{2 \pi \log j}/t_j} \to 0$. For $s_j < s < 1$, $\alpha t_j^2\, (\log s)^2/2 \pi \log j$ is bounded by $\alpha M_\eps^2$ and goes to zero as $j \to \infty$, so the last integral converges to
\[
\int_0^1 s^{1 - \gamma}\, ds = \frac{1}{2 - \gamma}.
\]
So combining \eqref{3.14}--\eqref{3.19} and letting $j \to \infty$ gives
\[
t_0^2 \ge (\beta - \eps)\, \frac{2 \pi d^{2 - \gamma}}{2 - \gamma}\, (L_1 + L_2 - 1),
\]
where
\begin{gather*}
L_1 = \liminf_{j \to \infty}\, j^{\alpha\, (t_j^2 - t_0^2)/2 \pi},\\[7.5pt]
L_2 = \liminf_{j \to \infty}\, \int_0^1 ne^{- nt\, [1 - (t_j/t_0)^2\, t]}\, dt,
\end{gather*}
and $n = (2 - \gamma) \log j \to \infty$. Letting $\eps \to 0$ in this inequality gives
\begin{equation} \label{3.20}
\beta \le \frac{2 \kappa}{\alpha\, (L_1 + L_2 - 1)}.
\end{equation}
We will show that this leads to a contradiction if ({\em a}) or ({\em b}) holds.

({\em a}) By \eqref{1.6}, $G(t_j \omega_j) \ge 0$ and hence \eqref{3.9} gives $t_j \ge t_0$, so $L_1 \ge 1$ and
\[
L_2 \ge \lim_{n \to \infty}\, \int_0^1 ne^{- nt\, (1 - t)}\, dt = 2
\]
(see de Figueiredo et al.\! \cite{MR1386960, MR1399846}). Then \eqref{3.20} gives $\beta \le \kappa/\alpha$, contradicting \eqref{1.8}.

({\em b}) By \eqref{1.9}, $G(t_j \omega_j) \ge - \sigma_0\, t_j^2 \omega_j^2/2$ and hence \eqref{3.9} and \eqref{3.4.2} give
\[
t_j^2 - t_0^2 \ge - \sigma_0\, t_j^2 \int_\Omega \frac{\omega_j^2}{|x|^\gamma}\, dx \ge - \frac{\sigma_0\, t_j^2}{\kappa\, (2 - \gamma) \log j} = - \frac{\sigma_0\, t_j^2}{\kappa n},
\]
so
\[
t_j^2 - t_0^2 \ge - \frac{\sigma_0\, t_0^2}{\kappa\, (2 - \gamma) \log j + \sigma_0} \ge - \frac{2 \pi \sigma_0}{\alpha \kappa \log j}
\]
and
\[
\left(\frac{t_j}{t_0}\right)^2 \ge \frac{\kappa n}{\kappa n + \sigma_0} \ge 1 - \frac{\sigma_0}{\kappa n}.
\]
So
\[
L_1 = \liminf_{j \to \infty}\, e^{\alpha\, (t_j^2 - t_0^2) \log j/2 \pi} \ge e^{- \sigma_0/\kappa}
\]
and
\[
L_2 \ge \liminf_{n \to \infty}\, \int_0^1 ne^{- nt\, (1 - t) - \sigma_0\, t^2/\kappa}\, dt \ge e^{- \sigma_0/\kappa} \left(\lim_{n \to \infty}\, \int_0^1 ne^{- nt\, (1 - t)}\, dt\right) = 2e^{- \sigma_0/\kappa}.
\]
On the other hand,
\[
L_2 \ge \lim_{n \to \infty}\, \int_0^1 ne^{- nt}\, dt = 1.
\]
Then \eqref{3.20} gives
\[
\beta \le \frac{2 \kappa}{\alpha\, (e^{- \sigma_0/\kappa} + \max \set{2e^{- \sigma_0/\kappa}, 1} - 1)},
\]
contradicting \eqref{1.10}.
\end{proof}

We can now conclude the proofs of Theorems \ref{Theorem 1.1} and \ref{Theorem 1.2}. Let $j_0$ be as in Lemma \ref{Lemma 3.2} \ref{Lemma 3.2 (ii)}. By Lemma \ref{Lemma 3.2} \ref{Lemma 3.2 (i)}, $\exists R > \rho$ such that $E(R \omega_{j_0}) \le 0$, where $\rho$ is as in Lemma \ref{Lemma 3.1}. Let
\[
\Gamma = \set{\gamma \in C([0,1],H^1_0(\Omega)) : \gamma(0) = 0,\, \gamma(1) = R \omega_{j_0}}
\]
be the class of paths joining the origin to $R \omega_{j_0}$, and set
\[
c := \inf_{\gamma \in \Gamma}\, \max_{u \in \gamma([0,1])}\, E(u).
\]
By Lemma \ref{Lemma 3.1}, $c > 0$. Since the path $\gamma_0(t) = tR \omega_{j_0},\, t \in [0,1]$ is in $\Gamma$,
\[
c \le \max_{u \in \gamma_0([0,1])}\, E(u) \le \sup_{t \ge 0}\, E(t \omega_{j_0}) < \frac{2 \pi}{\alpha} \left(1 - \frac{\gamma}{2}\right)
\]
by \eqref{3.5}. If there are no \PS{c} sequences of $E$, then $E$ satisfies the \PS{c} condition vacuously and hence has a critical point $u$ at the level $c$ by the mountain pass theorem. Then $u$ is a solution of problem \eqref{1.1} and $u$ is nontrivial since $c > 0$. So we may assume that $E$ has a \PS{c} sequence. Then this sequence has a subsequence that converges weakly to a nontrivial solution of problem \eqref{1.1} by Proposition \ref{Proposition 2.1}.

\section{Proof of Theorem \ref{Theorem 1.3}} \label{Section 4}

In this section we prove Theorem \ref{Theorem 1.3} by showing that the functional $E$ has the linking geometry with the minimax level $c \in (0,2 \pi\, (1 - \gamma/2)/\alpha)$ and applying Proposition \ref{Proposition 2.1}. We have the direct sum decomposition
\[
H^1_0(\Omega) = V \oplus W, \quad u = v + w,
\]
where $V$ is the subspace spanned by the eigenfunctions of $\lambda_1(\gamma), \dots, \lambda_{k-1}(\gamma)$ and $W$ is the closure of the subspace spanned by the eigenfunctions of $\lambda_k(\gamma), \lambda_{k+1}(\gamma), \dots$. It follows from Proposition \ref{Proposition 2.2} that
\begin{equation} \label{4.1}
\int_\Omega |\nabla v|^2\, dx \le \lambda_{k-1}(\gamma) \int_\Omega \frac{v^2}{|x|^\gamma}\, dx \quad \forall v \in V
\end{equation}
and
\begin{equation} \label{4.2}
\int_\Omega |\nabla w|^2\, dx \ge \lambda_k(\gamma) \int_\Omega \frac{w^2}{|x|^\gamma}\, dx \quad \forall w \in W.
\end{equation}

\begin{lemma} \label{Lemma 4.1}
If \eqref{1.12} holds, then there exists a $\rho > 0$ such that
\[
\inf_{\substack{w \in W\\ \norm{w} = \rho}}\, E(w) > 0.
\]
\end{lemma}

\begin{proof}
As in the proof of Lemma \ref{Lemma 3.1} in the last section, \eqref{1.12} gives
\[
\int_\Omega \frac{G(w)}{|x|^\gamma}\, dx \le \half\, (\lambda_k(\gamma) - \sigma_2) \int_\Omega \frac{w^2}{|x|^\gamma}\, dx + C_\delta \int_\Omega |w|^3\, \frac{e^{\alpha w^2}}{|x|^\gamma}\, dx
\]
for some constant $C_\delta > 0$ and the last integral is $\O(\rho^3)$ as $\rho = \norm{w} \to 0$. Since
\[
\int_\Omega \frac{w^2}{|x|^\gamma}\, dx \le \frac{\rho^2}{\lambda_k(\gamma)}
\]
by \eqref{4.2}, then
\[
E(w) \ge \half\, \frac{\sigma_2}{\lambda_k(\gamma)}\, \rho^2 + \O(\rho^3) \quad \text{as } \rho \to 0,
\]
and the desired conclusion follows for sufficiently small $\rho > 0$.
\end{proof}

Let $\omega_j$ be as in \eqref{3.4}, and set
\[
Q_{j,\, R} = \set{u = v + t \omega_j : v \in V,\, t \ge 0,\, \norm{u} \le R}
\]
for $R > 0$.

\begin{lemma} \label{Lemma 4.2}
Assume that $h$ satisfies \eqref{1.1.1} and \eqref{1.2}.
\begin{enumroman}
\item \label{Lemma 4.2 (i)} If \eqref{1.11} holds, then for all $j \ge 2$, there exists a $R_j > 0$ such that
    \[
    \sup_{u \in \bdry{Q_{j,R}}}\, E(u) = 0 \quad \forall R \ge R_j.
    \]
\item \label{Lemma 4.2 (ii)} If \eqref{1.11} and \eqref{1.12} hold, then there exists a constant $c > 0$ depending only on $\Omega$, $\alpha$, $\gamma$, and $k$ such that whenever
    \[
    \beta > \frac{2 \kappa}{\alpha}\, e^{c/\sigma_0},
    \]
    there exists a $j_0 \ge 2$ such that
    \begin{equation} \label{4.3}
    \sup_{v \in V,\, t \ge 0}\, E(v + t \omega_{j_0}) < \frac{2 \pi}{\alpha} \left(1 - \frac{\gamma}{2}\right).
    \end{equation}
\end{enumroman}
\end{lemma}

\begin{proof}
As in the proof of Lemma \ref{Lemma 3.2}, for each $\eps > 0$, there exist constants $M_\eps, C_\eps > 0$ such that
\begin{equation} \label{4.3.1}
th(t)\, e^{\alpha t^2} > (\beta - \eps)\, e^{\alpha t^2} \quad \text{for } |t| > M_\eps
\end{equation}
and
\begin{gather}
\label{4.4} th(t)\, e^{\alpha t^2} \ge \half\, (\beta - \eps)\, \alpha^2\, t^4 - C_\eps\, |t|,\\[7.5pt]
\notag G(t) \ge \frac{1}{8}\, (\beta - \eps)\, \alpha^2\, t^4 - C_\eps\, |t|
\end{gather}
for all $t$. Then for all $u \ne 0$,
\[
E(tu) \le \frac{t^2}{2} \int_\Omega |\nabla u|^2\, dx - \frac{1}{8}\, (\beta - \eps)\, \alpha^2\, t^4 \int_\Omega \frac{u^4}{|x|^\gamma}\, dx + C_\eps\, t \int_\Omega \frac{|u|}{|x|^\gamma}\, dx \to - \infty
\]
as $t \to \infty$. On the other hand, by \eqref{1.11} and \eqref{4.1},
\[
E(v) \le \half \left(\int_\Omega |\nabla v|^2\, dx - \lambda_{k-1}(\gamma) \int_\Omega \frac{v^2}{|x|^\gamma}\, dx\right) \le 0 \quad \forall v \in V.
\]
Since $Q_{j,R}$ lies in a finite dimensional subspace, \ref{Lemma 4.2 (i)} follows.

We will show that if \eqref{1.11} and \eqref{1.12} hold, but \eqref{4.3} does not hold for any $j_0 \ge 2$, then there exists a constant $c > 0$ depending only on $\Omega$, $\alpha$, $\gamma$, and $k$ such that
\begin{equation} \label{4.5}
\beta \le \frac{2 \kappa}{\alpha}\, e^{c/\sigma_0}.
\end{equation}
We have
\[
\sup_{v \in V,\, t \ge 0}\, E(v + t \omega_j) \ge \frac{2 \pi}{\alpha} \left(1 - \frac{\gamma}{2}\right) \quad \forall j \ge 2.
\]
It follows from Lemma \ref{Lemma 4.1} and \ref{Lemma 4.2 (i)} that the above supremum is attained at some point $u_j = v_j + t_j \omega_j,\, v_j \in V,\, t_j > 0$ such that $E'(u_j) = 0$ on $\set{v + t \omega_j : v \in V,\, t \ge 0}$. Then
\begin{gather}
\label{4.6} E(u_j) = \half \norm{u_j}^2 - \int_\Omega \frac{G(u_j)}{|x|^\gamma}\, dx \ge \frac{2 \pi}{\alpha} \left(1 - \frac{\gamma}{2}\right),\\[7.5pt]
\label{4.7} E'(u_j)\, u_j = \norm{u_j}^2 - \int_\Omega u_j\, h(u_j)\, \frac{e^{\alpha u_j^2}}{|x|^\gamma}\, dx = 0.
\end{gather}
Since $\norm{\omega_j} = 1$ and $G(t) \ge 0$ for all $t$ by \eqref{1.11}, \eqref{4.6} gives
\[
\norm{v_j} + t_j \ge t_0,
\]
where
\[
t_0 = \sqrt{\frac{4 \pi}{\alpha} \left(1 - \frac{\gamma}{2}\right)}.
\]
First we will show that $t_j \to t_0$ and $v_j \to 0$ as $j \to \infty$.

Combining \eqref{4.6} with \eqref{1.11} and \eqref{4.1} gives
\begin{equation} \label{4.9}
t_j^2 + 2t_j \int_\Omega \nabla v_j \cdot \nabla \omega_j\, dx \ge t_0^2 + 2\, (\lambda_{k-1}(\gamma) + \sigma_0)\, t_j \int_\Omega \frac{v_j\, \omega_j}{|x|^\gamma}\, dx + \sigma_0 \int_\Omega \frac{v_j^2}{|x|^\gamma}\, dx.
\end{equation}
Since $\nabla \omega_j = 0$ outside $\set{d/j < |x| < d}$ and $\omega_j$ is harmonic in $\set{d/j < |x| < d}$,
\begin{multline}
\abs{\int_\Omega \nabla v_j \cdot \nabla \omega_j\, dx} = \abs{\int_{\bdry{\set{d/j < |x| < d}}} v_j\, \frac{\partial \omega_j}{\partial n}\, ds} \le \pnorm[\infty]{v_j} \int_{\bdry{\set{d/j < |x| < d}}} |\nabla \omega_j|\, ds\\[7.5pt]
= 2 \pnorm[\infty]{v_j} \sqrt{\frac{2 \pi}{\log j}},
\end{multline}
and
\begin{equation} \label{4.9.1}
\abs{\int_\Omega \frac{v_j\, \omega_j}{|x|^\gamma}\, dx} \le \pnorm[\infty]{v_j} \int_\Omega \frac{\omega_j}{|x|^\gamma}\, dx \le \frac{d^{2 - \gamma}}{(2 - \gamma)^2} \pnorm[\infty]{v_j} \sqrt{\frac{2 \pi}{\log j}}
\end{equation}
by \eqref{3.4.1}, where $\pnorm[\infty]{\cdot}$ denotes the $L^\infty(\Omega)$ norm. Combining \eqref{4.9}--\eqref{4.9.1} and noting that $\lambda_{k-1}(\gamma) + \sigma_0 < \lambda_k(\gamma)$ by \eqref{1.11} and \eqref{1.12} gives
\begin{equation} \label{4.9.2}
\sigma_0 \pnorm[2,\, \omega]{v_j}^2 \le t_j^2 - t_0^2 + 2 \left(2 + \lambda_k(\gamma)\, \frac{d^{2 - \gamma}}{(2 - \gamma)^2}\right) t_j \pnorm[\infty]{v_j} \sqrt{\frac{2 \pi}{\log j}},
\end{equation}
where $\pnorm[2,\, \omega]{\cdot}$ is defined in \eqref{2.10}. Since $V$ is finite dimensional and $\sigma_0 > 0$, it follows from this that $\norm{v_j} = \O(t_j)$ and then
\begin{equation} \label{4.9.3}
\liminf_{j \to \infty}\, t_j \ge t_0.
\end{equation}

Next combining \eqref{4.7} with \eqref{4.3.1} and \eqref{4.4} gives
\begin{multline} \label{4.10}
\norm{u_j}^2 = \int_{\set{|u_j| > M_\eps}} u_j\, h(u_j)\, \frac{e^{\alpha u_j^2}}{|x|^\gamma}\, dx + \int_{\set{|u_j| \le M_\eps}} u_j\, h(u_j)\, \frac{e^{\alpha u_j^2}}{|x|^\gamma}\, dx\\[7.5pt]
\ge (\beta - \eps) \int_{\set{|u_j| > M_\eps}} \frac{e^{\alpha u_j^2}}{|x|^\gamma}\, dx - C_\eps \int_{\set{|u_j| \le M_\eps}} \frac{|u_j|}{|x|^\gamma}\, dx.
\end{multline}
For $|x| \le d/j$,
\[
|u_j| \ge t_j \omega_j - |v_j| \ge t_j\, \sqrt{\frac{\log j}{2 \pi}} - \pnorm[\infty]{v_j},
\]
and the last expression is greater than $M_\eps$ for all sufficiently large $j$ since $\pnorm[\infty]{v_j} = \O(t_j)$ and $\liminf t_j > 0$, so
\begin{multline*}
\int_{\set{|u_j| > M_\eps}} \frac{e^{\alpha u_j^2}}{|x|^\gamma}\, dx \ge e^{\alpha\, (t_j \sqrt{\log j/2 \pi} - \pnorm[\infty]{v_j})^2} \int_{\set{|x| \le d/j}} \frac{dx}{|x|^\gamma}\\[7.5pt]
= \frac{2 \pi}{2 - \gamma} \left(\frac{d}{j}\right)^{2 - \gamma} \hspace{-5pt} j^{\alpha\, (t_j - \pnorm[\infty]{v_j} \sqrt{2 \pi/\log j})^2/2 \pi} = \frac{2 \pi d^{2 - \gamma}}{2 - \gamma}\, j^{\alpha\, [(t_j - \pnorm[\infty]{v_j} \sqrt{2 \pi/\log j})^2 - t_0^2]/2 \pi}
\end{multline*}
for large $j$. So \eqref{4.10} gives
\begin{equation} \label{4.15}
(\beta - \eps)\, j^{\alpha\, [(t_j - \pnorm[\infty]{v_j} \sqrt{2 \pi/\log j})^2 - t_0^2]/2 \pi} \le \frac{2 - \gamma}{2 \pi d^{2 - \gamma}} \left[\left(\norm{v_j} + t_j\right)^2 + C_\eps \int_\Omega \frac{|v_j|}{|x|^\gamma}\, dx + \delta_j\, t_j\right],
\end{equation}
where
\[
\delta_j = C_\eps \int_\Omega \frac{\omega_j}{|x|^\gamma}\, dx \to 0 \quad \text{as } j \to \infty
\]
by \eqref{3.4.1}. Since $\norm{v_j} = \O(t_j)$, it follows from this that
\[
\limsup_{j \to \infty}\, t_j \le t_0,
\]
which together with \eqref{4.9.3} shows that $t_j \to t_0$. Then \eqref{4.9.2} implies that $v_j \to 0$.

Now the right-hand side of \eqref{4.15} goes to $2 \kappa/\alpha$ as $j \to \infty$. If $\beta \le 2 \kappa/\alpha$, then we may take any $c > 0$, so suppose that $\underline{\beta} > 2 \kappa/\alpha$. Then for $\eps < \underline{\beta} - 2 \kappa/\alpha$ and all sufficiently large $j$, \eqref{4.15} gives $j^{\alpha\, [(t_j - \pnorm[\infty]{v_j} \sqrt{2 \pi/\log j})^2 - t_0^2]/2 \pi} \le 1$ and hence
\[
t_j \le t_0 + \pnorm[\infty]{v_j} \sqrt{\frac{2 \pi}{\log j}}.
\]
Combining this with \eqref{4.9.2} gives
\[
\norm{v_j} \le \frac{c_1}{\sigma_0\, \sqrt{\log j}}, \qquad t_j^2 - t_0^2 \ge - \frac{c_2}{\sigma_0 \log j}
\]
for some constants $c_1, c_2 > 0$ depending only on $\Omega$, $\alpha$, $\gamma$, and $k$. Then $\alpha\, [(t_j - \pnorm[\infty]{v_j} \sqrt{2 \pi/\log j})^2 - t_0^2]/2 \pi \ge - c/\sigma_0 \log j$ for some constant $c > 0$ depending only on $\Omega$, $\alpha$, $\gamma$, and $k$, so \eqref{4.15} gives
\[
(\beta - \eps)\, e^{- c/\sigma_0} \le \frac{2 \kappa}{\alpha},
\]
and letting $\eps \to 0$ gives \eqref{4.5}.
\end{proof}

We can now conclude the proof of Theorem \ref{Theorem 1.3}. Let $c > 0$ be as in Lemma \ref{Lemma 4.2} \ref{Lemma 4.2 (ii)}, let $\beta$ $\beta$ satisfy \eqref{1.13}, and let $j_0$ also be as in Lemma \ref{Lemma 4.2} \ref{Lemma 4.2 (ii)}. By Lemma \ref{Lemma 4.2} $\exists R > \rho$ such that
\[
\sup_{u \in \bdry{Q_{j_0,R}}}\, E(u) = 0,
\]
where $\rho$ is as in Lemma \ref{Lemma 4.1}. Let
\[
\Gamma = \set{\gamma \in C(Q_{j_0,R},H^1_0(\Omega)) : \restr{\gamma}{\bdry{Q_{j_0,R}}} = \id},
\]
and set
\[
c := \inf_{\gamma \in \Gamma}\, \max_{u \in \gamma(Q_{j_0,R})}\, E(u).
\]
By Lemma \ref{Lemma 4.1}, $c > 0$. Since the identity mapping is in $\Gamma$,
\[
c \le \max_{u \in Q_{j_0,R}}\, E(u) \le \sup_{v \in V,\, t \ge 0}\, E(v + t \omega_{j_0}) < \frac{2 \pi}{\alpha} \left(1 - \frac{\gamma}{2}\right)
\]
by \eqref{4.3}. If there are no \PS{c} sequences of $E$, then $E$ satisfies the \PS{c} condition vacuously and hence has a critical point $u$ at the level $c$ by the linking theorem. Then $u$ is a solution of problem \eqref{1.1} and $u$ is nontrivial since $c > 0$. So we may assume that $E$ has a \PS{c} sequence. Then this sequence has a subsequence that converges weakly to a nontrivial solution of problem \eqref{1.1} by Proposition \ref{Proposition 2.1}.

\def\cdprime{$''$}

\end{document}